\documentclass[11pt]{article}

\usepackage{graphicx}
\usepackage[breaklinks=true, linktocpage=true,]{hyperref}
\usepackage{amssymb}
\usepackage{tikz}
\usepackage{commath}
\usepackage{braket}
\usepackage{wasysym}
\usepackage[margin=1in]{geometry}
\usepackage{booktabs}
\usepackage{amsmath}
\usepackage{amsthm}
\usepackage{lipsum}
\usepackage{url}
\usepackage[boxed]{algorithm2e}
\usepackage{algpseudocode}
\usepackage[capitalize,nameinlink]{cleveref}

\SetAlCapSkip{0.5em}

\newcommand{\CC}{{\cal C}}
\renewcommand{\phi}{\varphi}

\hypersetup{
	colorlinks = true,
    linkcolor={purple},
    urlcolor={purple},
    citecolor={blue!80!black}
}

\newtheorem*{theorem*}{Theorem}
\newtheorem{definition}{Definition}
\newtheorem{theorem}[definition]{Theorem}
\newtheorem{lemma}[definition]{Lemma}

\newtheorem{proposition}[definition]{Proposition}

\newtheorem{corollary}[definition]{Corollary}

\newtheorem{remark}[definition]{Remark}



\newcommand{\etalchar}[1]{$^{#1}$}

\newcommand{\abra}[1]{{\left\langle #1 \right\rangle}}

\newcommand{\cbra}[1]{{\left\{ #1 \right\}}}

\newcommand{\F}{\mathbb{F}}

\newcommand{\eps}{\varepsilon}
\renewcommand{\epsilon}{\eps}

\newcommand{\sse}{\subseteq}

\renewcommand{\Pr}{\operatorname{{\bf Pr}}}

\usepackage{tikz}
\usepackage{verbatim}
\usepackage{cancel}
\usetikzlibrary{shapes.geometric}
\usetikzlibrary{patterns}

\newcommand{\ignore}[1]{}

\title{Sumsets in the Hypercube \vspace{0.5em}}
\author{
Noga Alon
\thanks{Princeton University,
Princeton, NJ, USA
and
Tel Aviv University, Tel Aviv,
Israel.
Email: {\tt nalon@math.princeton.edu}.
Research supported in part by
NSF grant DMS-2154082.}
\and 
Or Zamir
\thanks{Tel Aviv University, Tel Aviv, Israel.
Email: {\tt orzamir90@gmail.com}
\vspace{0.5em}
}}
\date{}

\begin{document}

\pagenumbering{gobble}
\maketitle  

\begin{abstract}
    A subset $S$ of the Boolean hypercube $\F_2^n$ is a \emph{sumset} if $S = A+A = \{a + b \ | \  a, b\in A\}$ for some $A \subseteq \F_2^n$. 
    We prove that the number of sumsets in~$\F_2^n$ is asymptotically
	$(2^n-1)2^{2^{n-1}}$.
    Furthermore, we show that the family of sumsets in~$\F_2^n$ is 
	almost identical to the family of all subsets of~$\F_2^n$ 
	that contain a complete linear subspace of co-dimension~$1$.
\end{abstract}

\pagenumbering{arabic}

\section{Introduction}
\label{sec:intro}

A subset $S$ of an Abelian group~$G$ is a \emph{sumset} if $S = A+A = \{a + b \ | \  a, b\in A\}$ for some $A \subseteq G$.
Sumsets are among the most fundamental objects studied in \emph{additive combinatorics} (for a comprehensive survey, see the book of Tao and Vu~\cite{tao2006additive}).
When~$G$ is finite, then arguably the simplest question regarding sumsets in~$G$ is how many distinct ones there are. A classical work of Green and Ruzsa~\cite{green2004counting} and later refinements~\cite{sargsyan2015counting} study this question for~$G=\F_p$.
Other well studied counting questions with regards to sumsets include estimating the number of different sums~$A+B$ when the sizes of~$|A|$ and~$|B|$ are both large~\cite{alon2010number, sapozhenko2020number}; 
or estimating the number of sets~$A$ with a small sumset~$|A+A|\leq K|A|$~\cite{freiman1973foundations, green2005counting, alon2014refinement, green2016counting}.

Most of the results cited above study the case of~$G=\F_p$ for some large prime~$p$.
In this work, we focus on the Boolean hypercube~$G=\F_2^n$, that is, the
vector space of dimension $n$ over $F_2$.
This choice naturally comes up in applications of additive combinatorics to theoretical computer science (see for example the surveys of~\cite{additive-comb-minicourse,Trevisan09,Viola11ac,Bibak13,Lovett17actcs}).
Furthermore, vectors spaces~$\F_p^n$ over finite fields (also called finite field models) are often easier to study due to the availability of tools from linear algebra~\cite{Green05}.
Recent breakthroughs in additive combinatorics also focus on finite field models:
Kelley and Meka's bounds for~$3$-progressions~\cite{kelley2023strong} 
are proven starting with an analysis in~$\F_p^n$; Gowers, Green, Manners and Tao prove the polynomial Freiman--Ruzsa conjecture in~$\F_2^n$~\cite{gowers2023conjecture}.

Put~$N:=2^n$ and denote the family of sumsets in~$\F_2^n$ by~$\mathcal{S}_n := \{ A+A \ | \ A\subseteq \F_2^n\}.$
Our main focus in this work is to understand 
the cardinality of~$\mathcal{S}_n$ and the typical structure of an element
in it.
It was shown in~\cite{sargsyan2015counting} that~$|\mathcal{S}_n| = 2^{N/2 + o(N)}$.
In a recent work on property testing of sumsets~\cite{chen2024testing}, this bound was improved to show that~$|\mathcal{S}_n|$ is between~$2^{N/2}$ and~$2^{N/2+n^2+O(1)}$.
The following theorem  provides an asymptotic 
formula for~$|\mathcal{S}_n|$. Here and in what follows the notation
$f \sim g$ for two functions $f$ and $g$ denotes
that $f=(1+o(1))g$ where the $o(1)$-term tends to $0$ as the parameters
of the functions grow to infinity. Equivalently, this means that the
limit of the ratio $f/g$ as the parameters grow is $1$.

\begin{theorem}\label{thm:asymptotic}
	$|\mathcal{S}_n| \sim (2^n-1)2^{N/2}$.
\end{theorem}

Furthermore, we provide a simple characterization of almost all sumsets in~$\F_2^n$.
Denote the family of subsets of~$\F_2^n$ that contain a complete linear subspace of co-dimension~$1$ by~$$\mathcal{H}_n := \{S\subseteq \F_2^n \ | \ \exists\; v \in \F_2^n \ . \ v^{\perp} \subseteq S \}.$$
We show that almost all sets in~$\mathcal{S}_n$ are also in~$\mathcal{H}_n$ and vice versa.

\begin{theorem}\label{thm:nearlyidentical}
    $|\mathcal{S}_n \Delta \mathcal{H}_n| = o\left(|\mathcal{S}_n|\right)$.
\end{theorem}

The containment of large linear subspaces within sumsets was studied beforehand. As a generalization to a question of Bourgain on arithmetic progressions in sums of sets of integers~\cite{bourgain1990arithmetic},
Green asked whether for every large~$A\subseteq\F_2^n$, the sumset~$A+A$ 
must contain a large linear subspace~\cite{Green05}.
It is simple to observe that if~$|A|> \frac{1}{2} |\F_2^n|$, then~$A+A=\F_2^n$.
Sanders~\cite{sanders2011green} showed that if~$|A| > \left(1/2 - \frac{1}{2^9 \sqrt{n}}\right)|\F_2^n|$ then~$A+A$ contains a subspace of co-dimension~$1$, or with our notation,~$A+A\in \mathcal{H}_n$. 
Other works (e.g.,~\cite{sanders2012bogolyubov}) continued studying the relation between the density of~$A$ and the largest linear subspace contained in~$A+A$.

Our lower bound to~$|\mathcal{S}_n|$ is rather straightforward.
For the upper bound, we need several structural claims on 
both sumsets in~$\F_2^n$ and independent sets in the~$n$-th 
dimensional hypercube~$Q_n$, that may be of independent interest.
As a byproduct to the discussion here, 
we show that almost all subsets of~$\F_2^n$ cannot be expressed
as unions of less than $O(N /(n^2 \log n))$ sumsets, and this is
tight up to a factor of $(\log n)^2$. This 
is somewhat counter-intuitive as cardinality-wise, even the family 
of pairs of sumsets is larger than the power set of the 
hypercube, that is, $|\mathcal{P}\left(\F_2^n\right)|<|\mathcal{S}_n|^2$.
The proof works for any finite abelian group, providing similar estimates.

\section{Notation}
\label{sec:prelims}
When the value of~$n$ is fixed and not ambiguous, we use~$N$ 
instead of~$2^n$ and~$A^c$ instead of~$\F_2^n\setminus A$.
We write 
$A + B := \cbra{ a+ b : a\in A, b\in B}.$
If one of the sets is a singleton, e.g. if $A = \{a\}$, we will sometimes write $a+B := \{a\} + B$ instead. 
We use~$v^\perp$ to denote the set of all vectors in~$\F_2^n$ orthogonal to~$v$.
As already mentioned, the notation~$f\sim g$ means that
$f=(1+o(1))g$, that is, $\lim (f/g) = 1$.

We write $H(x)$ to denote the binary entropy function $-x \log_2 x - (1-x) \log_2(1-x)$. Stirling's approximation gives the following helpful identity:
\begin{equation}\label{eq:stirling}
    \binom{n}{\alpha n} = \Theta^*(2^{H(\alpha)n}); \text{ or, equivalently,  }
    \binom{2^n}{\alpha 2^n} = 2^{H(\alpha)2^n} \cdot 2^{\Theta(n)}.
\end{equation}

Given a subset $D$ of an Abelian group $G$, we write $\Gamma_G(D)$ to denote the \emph{Cayley  (sum) graph} of $G$ with respect to the generator set $D$; that is, the graph on the vertex set $G$ that contains the edge $(x, y)$ if and only if $x + y \in D$. When $D=\{x\}$ is a singleton for some $x \in G$, we abuse notation slightly and write $\Gamma_G(x)$ for $\Gamma_G(\{x\})$.

\section{Simple Bounds} \label{sec:simple}
In this section we cover the bounds of~\cite{chen2024testing}, where it is shown that~$|\mathcal{S}_n|$ is between~$2^{N/2}$ and~$2^{N/2+n^2+O(1)}$.

\begin{lemma}\label{lem:cayley}
    Fix a sumset $A + A = S \subseteq \F_2^n$ and a set $D \subseteq \F_2^n$ with $S \cap D = \emptyset$. Then $A$ is an independent set in the Cayley graph $\Gamma_{\F_2^n}(D)$. 
\end{lemma}
\begin{proof}
    Assume for contradiction that~$A$ is not an independent set in~$\Gamma_{\F_2^n}(D)$. This implies the existence of $x, y \in A$ with $x + y = s$ for some $s \in D$. Hence, $s \in D \cap (A+A) = D \cap S$, which is a contradiction.
\end{proof}

\begin{proposition} \label{prop:sumset-count}
    The number of sumsets in $\F_2^n$ is at most 
    \[2^{2^{n-1} + n^2 + O(1)}.\]
\end{proposition}

\begin{proof}

Consider a sumset $S = A+A$; we consider two cases depending on the linear rank of the set $\F_2^n\setminus S$.

\medskip

\noindent{Case 1: $\F_2^n\setminus S$ does not have full rank.} In other words, there exists a vector $v\in\F_2^n$ such that 
\[\abra{x,v} = 1 \qquad\text{implies that}\qquad x \in S.\]
In particular, we have that~$S=S'\cup \left(\F_2^n \setminus v^\perp\right)$ for some $S' \sse v^\perp = \{x\in\F_2^n : \abra{x,v} = 0\}$. As there are at most $2^n$ choices for $v$, and for each choice of $v$ there are at most $2^{2^{n-1}}$ choices for $S'$, we have that there are at most $2^{2^{n-1} + n}$ many sumsets of this form.

\medskip

\noindent {Case 2: $\F_2^n\setminus S$ has full rank.} In particular, there are $n$ linearly independent vectors \emph{not} in $S$. As we have $n$ 
	linearly independent vectors not in $S$ it follows by 
	\cref{lem:cayley} that there exists a nonsingular 
transformation of $\F_2^n$ such that $A$ must be an independent set in 
	the graph of the hypercube~$Q_n$. As the number of independent sets in the hypercube~$Q_n$ is at most $2^{2^{n-1} + O(1)}$ (see for example \cite{numiss,galvin2019independent}), and as the number of nonsingular transformations 
	of the hypercube is at most $2^{n^2}$, it follows that the total number of sumsets of this form is at most 
\[2^{2^{n-1} + n^2 + O(1)}.\]

\medskip

Both cases together complete the proof.
\end{proof}

\begin{proposition}\label{prop:easylower}
    The number of sumsets in~$\F_2^n$ is at least~$2^{2^{n-1}}$.
\end{proposition}
\begin{proof}
    For any subset~$A\subseteq \F_2^{n-1}$ of the~$(n-1)$-th dimensional hypercube, we define a subset~$A'\subseteq \F_2^n$ as
    \[
    A':=\{\Vec{0}\} \cup \{(1,a) \ | \ a\in A\},
    \]
    where for a~$(n-1)$-dimensional vector~$a$, the concatenation~$(1,a)$ is defined as the~$n$-dimensional vector where the first coordinate is~$1$ and the other~$(n-1)$ coordinates are equal to~$a$.
    We observe that~$\left(A'+A'\right)\cap \left(\F_2^n \setminus e_1^{\perp}\right) = \{(1,a) \ | \ a\in A\}$. That is, in the sumset~$(A'+A')$ all vectors in which the first coordinate is~$1$ exactly correspond to the set~$A$. 
    In particular, for any~$A_1\neq A_2 \in \F_2^{n-1}$, we have~$(A'_1+A'_1)\neq (A'_2+ A'_2)$.
\end{proof}

\section{Lower Bound}\label{sec:lower}
In this section we improve \cref{prop:easylower}.
We prove that almost all subsets in~$\mathcal{H}_n$ are also sumsets, 
and that the size of~$\mathcal{H}_n$ is asymptotically~$2^{2^{n-1}}(2^n-1)$.
\begin{theorem}\label{thm:sizeh}
	$|\mathcal{H}_n| \sim 2^{2^{n-1}}(2^n-1)$.
\end{theorem}
\begin{theorem}\label{thm:lower}
    $|\mathcal{H}_n\setminus \mathcal{S}_n| = o\left(|\mathcal{H}_n|\right)$.
\end{theorem}

\begin{proof}[Proof of \cref{thm:sizeh}]
The set $\mathcal{H}_n$ is a union of the $(2^n-1)$ sets $H(v)$, 
	where for each nonzero vector $v$ in $\F_2^n$, $H(v)$ is the 
	family of all subsets containing $v^{\perp}$. As the size of each
	set $H(v) $ is $2^{2^{n-1}}$ this shows that 
	$|\mathcal{H}_n|\leq 2^{2^{n-1}}(2^n-1)$. 

	Since the intersection
	of each two distinct sets $H(u), H(v)$ is of cardinalty 
	$2^{2^{n-2}}$ it follows that
	$$|\mathcal{H}_n|\geq 2^{2^{n-1}}(2^n-1) 
	-\frac{(2^n-1)(2^n-2)}{2} \cdot 2^{2^{n-2}}  
	\geq 2^{2^{n-1}}(2^n-1) - 2^{2n-1+2^{n-2}} 
	= \left(1-o\left(1\right)\right)2^{2^{n-1}}(2^n-1).$$
\end{proof}

To prove \cref{thm:lower}, we use the following observation.
\begin{lemma}\label{lem:mostA}
    There are at most~$2^n \cdot 3^{2^{n-1}}$ subsets~$A\subseteq \F_2^n$ for which~$A+A \neq \F_2^n$.
\end{lemma}
\begin{proof}
    By \cref{lem:cayley}, if~$x\notin A+A$ then~$A$ is an independent set in~$\Gamma_{\F_2^n}(x)$, which is a perfect matching in~$\F_2^n$.
    Thus, there are at most~$3^{2^{n-1}}$ subsets~$A$ such that~$x\notin A+A$.
    We complete the proof by taking a union bound over all~$2^n$ choices for~$x$.
\end{proof}

\begin{proof}[Proof of \cref{thm:lower}]
    Take~$S\in \mathcal{H}_n$, and denote by~$v$ the vector such that~$v^\perp \subseteq S$.
    Denote by~$A := \{\vec{0}\} \cup \left(S \setminus v^\perp\right)$.
    Note that~$(A+A)\cap \left(\F_2^n \setminus v^\perp\right) = S\setminus v^\perp$.
    Thus, $A+A=S$ (and hence~$S\in \mathcal{S}_n$) if and only if~$(A+A)\cap v^\perp = v^\perp$.
    On the other hand,~$(A+A)\cap v^\perp = \{\vec{0}\} \cup \left(\left(S \setminus v^\perp\right)+\left(S \setminus v^\perp\right)\right)$.
    Denote by~$S' := \left(S \setminus v^\perp\right) - \{u\}$,
	where $u$ is, say, the lexicographically first vector not in 
	$v^{\perp}$. The shift by~$u$ does not change the sumset, and it is now 
	a subset of~$\F_2^n \cap v^\perp$ which is isomorphic 
	to the hypercube $\F_2^{n-1}$ of one dimension less.
    By \cref{lem:mostA} there are at most~$2^{n-1}\cdot 3^{2^{n-2}}$ such subsets~$S'$ for which~$S'+S'$ is not complete.
    We conclude that~$|\mathcal{H}_n\setminus \mathcal{S}_n| \leq 2^n \cdot 2^{n-1}\cdot 3^{2^{n-2}}=o\left(2^{2^{n-1}}\right)$.
\end{proof}

\section{Structural Tools for the Upper Bound}\label{sec:tools}
\subsection{Unions of Sumsets}\label{sec:unions}
In this section we prove an upper bound for the number of 
subsets of $\F_2^n$ that can be written as a union of 
at most $k$ sumsets. 
In the proof of the upper bound in Section~\ref{sec:upper}, 
we in fact use this statement only with regards to unions of two sumsets.
The proof we describe below works for any finite abelian group of order
$N$, and  implies that only a negligible number of subsets of such a group 
can be expressed
as a union of at most $O(N/\log^3 N)$ sumsets. As we describe later,
this can be improved to $O(N/(\log^2 N \log \log N))$, which is tight
up to a factor of $(\log \log N)^2$.
Since for our purpose here we only need
the case of unions of two sumsets we first describe a simple self-contained
proof of a weaker estimate that suffices. As done throughout the paper
here too we make no attempt to optimize the aboslute constants.
\begin{theorem}
\label{thm:unions}
Let $G$ be a finite abelian group of order $N$. Then for any integer
$s \geq 64 \log^2 N$, the number of subsets of $G$ that can be expressed
as a union of at most $k= \lfloor \frac{N}{2s \ln (eN/s)} \rfloor$
	sumsets is at most $2^{N-s/8}+e^{N/2}$.
\end{theorem}

Note that taking, say, $s=N/20$ it follows that the number of sets that
can be expressed as a union of two sumsets is (much) smaller than
$2^{N-N/200}$.

The proof is based on the approach used in the study of the typical
independence number of random Cayley (or Cayley sum) graphs
of abelian groups, see, for example, 
\cite{AO95,ALON20131232,chen2024testing}.
For completeness we
repeat (or slightly paraphrase) those arguments to derive 
an explicit tail bound.

\begin{definition}
We say that a set~$A\subseteq G$ 
	has \emph{many sums} if~$|A+A|\geq \frac{1}{4} |A|^2$.
\end{definition}

\begin{lemma}[Appears in~\cite{ALON20131232,chen2024testing}]
\label{alem:manysums}
Let~$A\subseteq G$ be a non-empty set, then there exists~$A'\subseteq A$ 
such that~$|A'|\geq \sqrt{|A|}$ and~$A'$ has many sums.
\end{lemma}
\begin{proof}
We construct~$A'$ iteratively in a greedy manner. 
Starting from~$A'=\{a\}$ for an arbitrary~$a\in A$, as long as there 
is any~$x\in A\setminus A'$ such that~$\left(A'+x\right)\setminus 
\left(A'+A'\right) \geq \frac{1}{2} \left(|A'|+1\right)$ 
then we add~$x$ to~$A'$.

Let~$a\in A'$.
We observe that~$|A'+A'|\leq \frac{|A'|\cdot \left(|A'|-1\right)}{2}$, 
and thus for a uniformly chosen~$x$ from~$A$ we have
$$\Pr[a+x \in A'+A']\leq \frac{|A'+A'|}{|A|} \leq 
\frac{|A'|\cdot \left(|A'|-1\right)}{2|A|}.
$$
Hence, for a uniformly chosen~$x$ from~$A$ the expected number of elements~$a\in A'$ for which~$a+x\in A'+A'$ is at most~$|A'| \cdot \frac{|A'|\cdot \left(|A'|-1\right)}{2|A|} = \frac{|A'|-1}{2} \cdot \frac{|A'|^2}{|A|}$. 
    In particular, there exists such an outcome of~$x\in A$.
    If~$|A'|\leq \sqrt{|A|}$ then~$x+a\in A'+A'$ for at most~$\frac{|A'|-1}{2}$ elements~$a\in A'$. Each of the sums~$(x+a)$, for all~$a\in A'$, is unique as~$x$ is fixed. Thus, $\left(A'+x\right)\setminus \left(A'+A'\right) \geq \frac{1}{2} \left(|A'|+1\right)$.
    We conclude that the greedy process would not halt before~$A'$ is of size at least~$\sqrt{|A|}$.
    
    Finally, note that by definition of the greedy process we have
    $$
    |A'+A'| \geq \frac{1}{2}\left(1+2+\ldots+|A'|\right) \geq \frac{1}{4} |A'|^2
    .$$
\end{proof}

\begin{lemma}[Adapted from~\cite{ALON20131232,chen2024testing}]
	\label{alem:manysumsIS}
Let $G$ be a finite abelian group of order $N$ and let $S$ be a uniformly
random subset of $G$. For any integer $s \geq 64 \log^2 N$,
the probability that $S$ contains a sumset $A+A$ for
a set $A$ of size $s$
is at most $2^{-s/8}$.
\end{lemma}
\begin{proof}
For any fixed subset $A'$ of size $\sqrt s$ that has many sums 
	the probability that $A'+A'$ is contained in $S$ is at most
	$2^{-s/4}$. Therefore, by a union bound, the probability
	that there is such an $A'$ is at most 
$$
	{N \choose {\sqrt s}} 2^{-s/4} \leq 2^{-s/8},
	$$
	where here we used the fact that $s \geq 64 \log^2 N$.
	By the previous lemma, if there is no such $A'$ then there
	can be no $A$ of size $s$ so that $A+A \subset S$ , completing
	the proof.
\end{proof}

\begin{proof}[Proof of \cref{thm:unions}]
	Let $S$ be a uniform random subset of $G$. By the last lemma,
	the probability that there is a subset of size $s$ whose sumset
	is in $S$ is at most $2^{-s/8}$. Therefore, the number of
	such subsets is at most $2^{N-s/8}$. For any other choice of the
	set $S$, if $S$ can be expressed as a union of sumsets, then each 
	such sumset~$A+A$ has $|A|<s$. The number of
	choices of such a sumset is at most $\sum_{i < s} {N \choose i}
	\leq (eN/s)^s$. The number of possible unions of $k$ such
	sumsets is thus at most $(eN/s)^{sk} \leq e^{N/2}$,
	by the choice of $k$. This completes the proof.
\end{proof}

\subsection{Parity Balance of Independent Sets in the Hypercube}\label{sec:balance}
Korshunov and Sapozhenko~\cite{numiss} proved the following tight bound on the number of independent sets in the hypercube.
See also Galvin's exposition of their proof~\cite{galvin2019independent}.
\begin{theorem}[\cite{numiss}]\label{thm:number_is}
    $i(Q_n) \sim 2\sqrt{e}\cdot 2^{N/2}$.
\end{theorem}
In almost all such independent sets, all but very 
	few of the vertices have the same Hamming weight parity.
	Let ~$\mathcal{E}$ denote the set of vertices of even weight
	of $Q_n$ and let $\mathcal{O}$ denote the set of vertices
	with odd weight.
As both~$\mathcal{E}$ and~$\mathcal{O}$ are independent sets in~$Q_n$ of size~$N/2$, it follows that there are~$2\cdot 2^{N/2} - 1$ independent sets in which all vertices have the same parity. This turns out to not be too far off from the total number of independent sets.
For a small~$k$, assume that~$|I\cap \mathcal{E}| = k$.
There are~${N/2 \choose k} \approx \frac{2^{k(n-1)}}{k!}$ choices for~$|I\cap \mathcal{E}|$, each of those is an independent set.
As~$Q_n$ is~$n$-regular, there are at most~$nk$ 
vertices of~$\mathcal{O}$ that neighbor any vertex of~$I\cap \mathcal{E}$.\footnote{We note that when~$k$ is small, this is also the typical number of such vertices.}
Hence, there are 
at least $N/2 - nk$ vertices in~$\mathcal{O}$ among which we may choose any subset to complete the independent set~$I$.
Therefore, the number of independent sets~$I$ with~$|I\cap \mathcal{E}| = k$ is at least~${N/2 \choose k} \cdot 2^{N/2-nk} \approx \frac{2^{k(n-1)}}{k!} \cdot 2^{N/2-nk} = \frac{2^{-k}}{k!} \cdot 2^{N/2}$.
By symmetry, we conclude that the number of independent sets in~$Q_n$ in which~$\min\{|I\cap \mathcal{E}|,\;|I\cap\mathcal{O}|\} \leq k$ is at least
$$
\left(1-o\left(1\right)\right)\cdot 2\cdot \sum_{i=0}^{k} \left(\frac{2^{-k}}{k!} \cdot 2^{N/2}\right) = 
2 \cdot \left(\sqrt{e} - o_k\left(1\right)\right) \cdot 2^{N/2}
.$$
As this simple lower bound is already matching the upper bound in \cref{thm:number_is}, we conclude the following.
\begin{corollary}\label{lem:simplerbalancedupper}
    Let~$k(n)=\omega_n(1)$ be any super-constant function in~$n$, then the number of independent sets~$I$ in~$Q_n$ in which~$\min\{|I\cap \mathcal{E}|,\;|I\cap\mathcal{O}|\} > k(n)$ is~$o\left(2^{N/2}\right)$.
\end{corollary}
\cref{lem:simplerbalancedupper} also follows formally from several works on the hardcore distributional model of the hypercube~\cite{kahn2001entropy, galvin2011threshold, jenssen2020independent} which with parameter~$\lambda = 1$ 
coincides with drawing a uniform independent set of~$Q_n$. 
For example, Theorem 1.4 clause 2 in~\cite{galvin2011threshold} proves that the distribution of~$\min\{|I\cap \mathcal{E}|,\;|I\cap\mathcal{O}|\}$ when~$I$ is a uniformly chosen independent set in~$Q_n$, converges to a Poisson distribution with parameter~$\frac{1}{2}$.
In this section we derive a simple \emph{tail bound} on the probability 
that both parities of a random independent set are of non-negligible size.
We note that a stronger estimate follows from the results 
of Jenssen and Perkins in 
\cite{jenssen2020independent}, and a nearly tight estimate 
of $2^{N/2-\Omega(N/\sqrt{n})}$ follows from  the proof of Park in
\cite{park2022note} (see equation (10) in her paper
and the two lines preceding it).
The bound below is weaker, but suffices (with room to spare) 
for our purpose here, and as 
its proof is simple and very different from the ones in the papers
above we include it.

\begin{theorem}\label{thm:balancedIS}
    For any constant~$\beta>0$, there are at most~$2^{N/2-\sqrt{N}/{2^{\Theta(\sqrt{n})}}}$ independent sets $I$ in~$Q_n$ for which~$\min\{|I\cap \mathcal{E}|,\;|I\cap\mathcal{O}|\} > \beta N.$
\end{theorem}

The proof is based on the simple estimate in
\cref{lem:simplerbalancedupper}, 
together with the recursive structure of the hypercube. These are combined
with a known result 
from the theory of VC-dimension, and the 
FKG Inequality, or its earlier versions due to Harris and Kleitman.

\begin{definition}
    We say that a subset~$J\subseteq[n]$ of the coordinates~$[n]=\{1,2,\ldots,n\}$ is \emph{shattered} by a family~$\mathcal{F}\subseteq \F_2^n$, if for every~$J'\subseteq J$ there exists~$F\in\mathcal{F}$ for which~$J'=J\cap F$.
\end{definition}
The following lemma is a variant of the Sauer-Perles-Shelah lemma
\cite{sauer1972density, shelah1972combinatorial}, due to Pajor
\cite{pajor1985sous}, see also \cite{anstee2002shattering}.
\begin{lemma}
\label{lem:ssp}
    Any family~$\mathcal{F}\subseteq \F_2^n$ shatters at least~$|\mathcal{F}|$ different subsets~$J\subseteq [n]$.
\end{lemma}
\begin{definition}
    We say that a family~$\mathcal{J}\subseteq \mathcal{P}\left([n]\right)$ is a \emph{down-closed family} if for any~$J\in \mathcal{J}$ and any~$J'\subseteq J$ we have that also~$J'\in \mathcal{J}$.
\end{definition}
The following is a special case of the Fortuin–Kasteleyn–Ginibre (FKG) 
inequality~\cite{fortuin1971correlation}, first proved by
Harris and by Kleitman~\cite{harris1960lower, kleitman1966families}.
\begin{lemma}[FKG]\label{lem:fkg}
    Let~$\mathcal{J},\mathcal{K}\subseteq\mathcal{P}\left([n]\right)$ be two down-closed families.
    Then,~$\frac{1}{N}|\mathcal{J}\cap \mathcal{K}|\geq \frac{1}{N}|\mathcal{J}| \cdot \frac{1}{N}|\mathcal{K}|$.
\end{lemma}

We also use the following simple application of the Chernoff–Hoeffding inequality~\cite{hoeffding1994probability}.
\begin{lemma}\label{lem:chern}
    Let~$\mathcal{J}\subseteq \mathcal{P}\left([n]\right)$ be a family of size~$|\mathcal{J}|>\gamma N$, then there exists~$J\in\mathcal{J}$ of size~$|J|>\frac{1}{2}n-\sqrt{\frac{1}{2} \ln \left(1/\gamma\right)}\cdot \sqrt{n}$.
\end{lemma}
\begin{proof}
    Let~$x$ be random vector drawn uniformly from~$\F_2^n$.
    By the Chernoff–Hoeffding inequality, the probability that the Hamming weight of~$x$ is at most~$n-x\sqrt{n}$ is smaller than~$e^{-2x^2}$.
    In particular, there are less than~$\gamma N$ different vectors in~$\F_2^n$ of weight at most~$\frac{1}{2}n-\sqrt{\frac{1}{2} \ln \left(1/\gamma\right)}\cdot \sqrt{n}$.
    The desired result follows by identifying each subset 
of~$[n]$ with its corresponding characteristic vector.
\end{proof}

We are now ready to prove \cref{thm:balancedIS}.
\begin{proof}[Proof of \cref{thm:balancedIS}]
    Let~$I$ be an independent set in~$Q_n$ for which $\min\{|I\cap \mathcal{E}|,\;|I\cap\mathcal{O}|\} > \beta N$.
    For any vertex~$x\in \F_2^n$ we define~$S_x := \{i\in [n] \ | \ x_i=1\}\subseteq[n]$ to be the set of coordinates~$i$ in which~$x_i=1$.
    Denote by~$\mathcal{F}_\mathcal{E} := \{S_x \ | \ x\in I \cap \mathcal{E}\}$, and respectively~$\mathcal{F}_\mathcal{O} := \{S_x \ | \ x\in I \cap \mathcal{O}\}$.
    By our assumption, we have~$|\mathcal{F}_\mathcal{E}|,|\mathcal{F}_\mathcal{O}| > \beta N$.
    Denote by~$\mathcal{J}_\mathcal{E}\subseteq \mathcal{P}\left([n]\right)$ (respectively,~$\mathcal{J}_\mathcal{O}$) the family of all subsets~$J\subseteq[n]$ of coordinates such that~$\mathcal{F}_\mathcal{E}$ (respectively,~$\mathcal{F}_\mathcal{O}$) shatters~$J$.
    By \cref{lem:ssp},~$|\mathcal{J}_\mathcal{E}|, |\mathcal{J}_\mathcal{O}| > \beta N$.
    We note that if a family shatters~$J$ then it also shatters any~$J'\subseteq J$, and hence~$\mathcal{J}_\mathcal{E},\mathcal{J}_\mathcal{O}$ are down-closed families.
    Denote by~$\mathcal{J}:= \mathcal{J}_\mathcal{E}\cap \mathcal{J}_\mathcal{O}$ the family of all subsets of coordinates shattered by both~$\mathcal{F}_\mathcal{E}$ and~$\mathcal{F}_\mathcal{O}$, by \cref{lem:fkg} we thus have~$|\mathcal{J}| \geq \beta^2 N$.
    Using \cref{lem:chern}, we conclude there exists a set of coordinates~$J\in \mathcal{J}$ that is shattered by both~$\mathcal{F}_\mathcal{E}$ and~$\mathcal{F}_\mathcal{O}$, of size~$|J|>\frac{1}{2}n-\sqrt{\ln \left(1/\beta\right)}\cdot \sqrt{n}$.
    Let~$K\subseteq J$ be an arbitrary subset of~$J$ of size, say,
	$|K|=\lceil \frac{1}{2}n-2\sqrt{\ln \left(1/\beta\right)} \cdot \sqrt{n} \rceil$.
    For any~$K'\subseteq K$, denote by~$I^{(K')}$ the subset of~$I$ containing only the vertices in which the coordinates of~$K$ exactly correspond to~$K'$, that is~$I^{(K')} := \{x\in I \ | \ S_x \cap K = K'\}$.
    We observe that~$I^{(K')}$ is an independent set in the~$\left(n-|K|\right)$-th dimensional hypercube defined by the same restriction~$\{x\in \F_2^n \ | \ S_x \cap K = K'\}\cong \F_2^{n-|K|}$.
    As~$J\supset K$ is shattered by both~$\mathcal{F}_\mathcal{E}$ and~$\mathcal{F}_\mathcal{O}$, we conclude that there are at least~$2^{|J|-|K|}$ vertices in~$I^{(K')}$ of each parity.
    In particular,~$I^{(K')}$ is an independent set in an~$\left(n-|K|\right)$-th dimensional hypercube with at least~$2^{|J|-|K|} = 2^{\Omega\left(\sqrt{\ln \left(1/\beta\right)}\cdot \sqrt{n}\right)}=\omega_n(1)$ vertices of each parity.
    By \cref{lem:simplerbalancedupper}, for a large enough~$n$ there are at most~$\frac{1}{2} \cdot 2^{\left(2^{n-|K|}\right)/2}$ possible such independent sets~$I^{(K')}$.
    Note that this family of possible independent sets only depends on~$n,\beta,K,K'$ and not on anything else (such as~$I$ itself or even~$J$).
    As~$I=\bigcup_{K'\subseteq K} I^{(K')}$, we may repeat the argument above for every~$K'$ and conclude that there are at most~$\left(\frac{1}{2} \cdot 2^{\left(2^{n-|K|}\right)/2}\right)^{2^{|K|}} = 2^{-2^{|K|}} \cdot 2^{N/2}$ possible such independent sets~$I$, and that family of possible sets only depends on~$n,\beta,K$.
    By the pigeonhole principle, at least a~$2^{-n}$ fraction of the independent sets~$I$ in~$Q_n$ for which $\min\{|I\cap \mathcal{E}|,\;|I\cap\mathcal{O}|\} > \beta N$ will result in the same subset~$K\subseteq[N]$ in the above argument.
    Thus, the number of such independent sets is no more than~$2^n \cdot 2^{-2^{|K|}} \cdot 2^{N/2}$.
    We finally note that for any constant~$\beta>0$,
    $$
    2^n \cdot 2^{-2^{|K|}} < 
    2^n \cdot 2^{-2^{n/2-\sqrt{\ln \left(1/\beta\right)} \cdot \sqrt{n} / 2}} = 
    2^{-\sqrt{N}/{2^{\Theta(\sqrt{n})}}}
    .$$
\end{proof}

As mentioned before Theorem \cref{thm:balancedIS}, 
Park~\cite{park2022note} proved recently that there are 
asymptotically~$2^{\left(1-\Theta\left(1/{\sqrt{n}}\right)\right)N/2}$ 
independent sets~$I$ of~$Q_n$ in 
which~$|I\cap\mathcal{E}|=|I\cap\mathcal{O}|$.
Therefore, even when~$\beta>0$ is a constant, the bound
in \cref{thm:balancedIS} cannot be improved to anything 
below~$2^{N/2 - N/\Theta\left(\sqrt{n}\right)}$, and is thus 
inherently of the form~$2^{N/2 - o(N)}$.

\section{Upper Bound}\label{sec:upper}
In this section, we are finally ready to prove the upper bound on~$|\mathcal{S}_n|$.
\begin{theorem}\label{thm:upper}
    $|\mathcal{S}_n\setminus \mathcal{H}_n| = 2^{-\sqrt{N}/{2^{\Theta(\sqrt{n})}}}\cdot |\mathcal{S}_n| = o\left(|\mathcal{S}_n|\right)$.
\end{theorem}

We first use \cref{sec:unions} to prove the following lemma.
\begin{definition}
Denote by $\mathcal{H'}_n := \{S\subseteq \F_2^n \ | \ \exists\; v \in \F_2^n \ . \ \left(\F_2^n \setminus v^{\perp}\right) \subseteq S \}$ the family of all subsets of~$\F_2^n$ that contain the full complement of co-dimension~$1$ linear subspace. 
\end{definition}
\begin{lemma}\label{lem:hnprime}
	$|\mathcal{S}_n\cap \mathcal{H}'_n| = 2^{N/2 - \Omega(N)}$.
\end{lemma}
\begin{proof}
    Let~$S\in \mathcal{S}_n\cap \mathcal{H}_n'$.
    As~$S\in \mathcal{S}_n$, there exists some~$A\subseteq \F_2^n$ for which~$S=A+A$.
    As~$S\in \mathcal{H}'_n$, there exists some~$v\in \F_2^n$ such that~$\left(\F_2^n \setminus v^{\perp}\right) \subseteq S$.
    Denote by~$A_0 = A \cap v^\perp$, and by~$A_1 = A\setminus v^\perp$.
    We have~$A=A_0 \cup A_1$ and moreover,~$(A+A)\cap v^\perp = (A_0+A_0)\cup (A_1+A_1)$.
    We note that~$\F_2^n \cap v^\perp \cong \F_2^{n-1}$ is isomorphic to a~$(n-1)$-th dimensional hypercube. By \cref{thm:unions} the number 
	of possible unions of two sumsets in~$\F_2^{n-1}$ 
	is~$2^{2^{n-1}-\Omega\left(2^{n-1}\right)}$.
    As~$S$ is fully described by~$v$ and by such a union, we conlcude that the number of possible sets~$S$ is at most
    $$
    2^n \cdot 2^{2^{n-1}-\Omega({2^{n-1}})} =
	2^{N/2 - \Omega(N) }
    .$$
\end{proof}

We next use \cref{sec:balance} to prove the following lemma.
\begin{lemma}\label{lem:addedge}
    Let~$v\in \F_2^n$ be a vector of even Hamming weight.
    The number of independent sets in~$Q_n \cup \Gamma_{\F_2^n}(v)$ is at most~$2^{N/2-\sqrt{N}/{2^{\Theta(\sqrt{n})}}}$.
\end{lemma}
\begin{proof}
By \cref{thm:balancedIS} there are at most~$2^{N/2-\sqrt{N}/{2^{\Theta(\sqrt{n})}}}$ independent sets~$I$ in~$Q_n$ for which~$\min\{|I\cap \mathcal{E}|,\;|I\cap\mathcal{O}|\} > \frac{1}{200} N$.
On the other hand, the induced subgraph of~$\Gamma_{\F_2^n}(v)$ on~$\mathcal{E}$ (or~$\mathcal{O}$) is a perfect matching, and hence it contains only~$3^{N/4}$ independent sets.
Hence, the number of independent sets~$I$ of~$Q_n \cup \Gamma_{\F_2^n}(v)$ in which~$\min\{|I\cap \mathcal{E}|,\;|I\cap\mathcal{O}|\} \leq \frac{1}{200} N$ is at most~$$2\cdot 3^{N/4} \cdot 2^{H(1/100)\cdot N/2 + \Theta(n)}=O\left(1.95^{N/2}\right).$$
\end{proof}
\begin{remark}
    A bound of the form~$2^{N/2-o(N)}$ is the best possible in \cref{lem:addedge}. 
    Assume~$n$ is even and consider the even-weight edge~$v:=\vec{1}=(1,1,\ldots,1)$.
    The set of all even-weight vectors of weight~$<n/2$ and all odd-weight vectors of weight~$>n/2$ is of size~$\left(1-\Theta\left(\frac{1}{\sqrt{n}}\right)\right)\frac{N}{2}$ and is an independent set in~$Q_n \cup \Gamma_{\F_2^n}(v)$ (hence so is any subset of it).
\end{remark}

\begin{proof}[Proof of \cref{thm:upper}]
    Let~$S\in \mathcal{S}_n\setminus \mathcal{H}_n$.
    If~$S^c$ is not of full linear rank, then~$S\in \mathcal{H}'_n$.
    By \cref{lem:hnprime} the number of such sumsets~$S$ is small enough and thus we may assume that there exists a linear basis~$v_1,\ldots,v_n$ in~$S^c$.
    We apply the linear transformation mapping this basis to the standard one, and then observe using \cref{lem:cayley} that the set~$A$ for which~$S=A+A$ must be an independent set in~$Q_n$ (after the linear transformation, which we from now on work under the application of).
    For any even-weight edge~$v$, we know from \cref{lem:addedge} that there are at most~$2^{N/2-\sqrt{N}/{2^{\Theta(\sqrt{n})}}}$ sets~$A$ that are both an independent set of~$Q_n$ and also have~$v\notin (A+A)$ (equivalently,~$A$ is also an independent set in~$\Gamma_{\F_2^n}(v)$).
    By a union bound, there are only~$2^{n-1} \cdot 2^{N/2-\sqrt{N}/{2^{\Theta(\sqrt{n})}}}$ sets~$A$ that are both an independent set in~$Q_n$ and also have any even-weight edge~$v$ not in~$A+A$.
    If~$A$ is not of this form, then the linear subspace~$\left(\vec{1}\right)^\perp$ of all even-weight vectors is contained in~$A+A=S$, and thus~$S\in \mathcal{H}_n$ which is a contradiction. 
    We conclude that~$S$ can be described by a linear transformation, of which there are at most~$2^{n^2}$, and either a set in~$\mathcal{H}'_n$ or one of the~$2^{n-1} \cdot 2^{N/2-\sqrt{N}/{2^{\Theta(\sqrt{n})}}}$ sets described above.
    The overall number of possible sets~$S$ is thus~$2^{N/2-\sqrt{N}/{2^{\Theta(\sqrt{n})}}}$.
\end{proof}

The lower bound in \cref{thm:lower} and the upper bound in 
\cref{thm:upper} together conclude the proof of \cref{thm:nearlyidentical}.
In combination with \cref{thm:sizeh}, we also deduce \cref{thm:asymptotic}.

\section{Expressing random sets as unions of sumsets}

\cref{thm:unions} provides a lower bound on the minimum 
number of sumsets required to express a random subset of a finite
abelian group as their union.
We can in fact
improve the statement of this theorem and prove a nearly tight result,
as we show in this section. Throughout the section
we consider general finite abelian groups.
The arguments can be easily extended to the non-abelian case,
but to simplify the presentation 
we restrict the discussion to abelian groups.
It is convenient to adopt here the convention that in a sumset
$A+A$ we only include the sums of distinct elements of $A$. 
\begin{theorem}
	\label{tunion}
	Let $G$ be an abelian group of order $N$ and let $S$ be a random
	subset of $G$. Then
	\begin{itemize}
		\item
			With high probability $S$ cannot be expressed as a union
			of less than $\Omega( N/(\log^2 N \log \log N))$
			sumsets.
		\item
			With high probability $S$ is a union of at most
			$O( N \log \log N/(\log^2 N))$ sumsets.
	\end{itemize}
\end{theorem}

Any abelian
group of order $N$ provides a properly edge-colored
complete graph $K_N$ in which the vertices correspond to the 
group elements and every edge $ab$ is colored by the sum $a+b$.
A sumset $A+A$ is thus simply the set of all colors that appear
in the induced subgraph on $A$. The results described here apply
to general properly edge-colored graphs (which are far more general
than the colorings correspodning to groups). The following result 
is proved in a recent paper of Conlon, Fox, Pham and Yepremyan.
\begin{theorem}[\cite{cfpy}]
	Let $K$ be a properly edge-colored complete graph on $N$ vertices, 
	and let
	$S$ be a random subset of the colors obtained by picking
	each color, randomly and independently, with probability 
	$1/2$. Then with high probability, the maximum number of
	vertices of a clique in the graph consisting of all edges
	colored by colors in $S$ is at most $O( \log N \log \log N)$.
\end{theorem}
The first part of Theorem \ref{tunion} follows quickly from the result
above. Indeed, 
by this theorem, with high probability no sumset of a subset of size
larger than $q=O( \log N \log \log N)$ is contained
in a random subset $S$.  There are at most 
$\sum_{i=0}^q {N \choose i} \leq N^q$ such subsets, so the number of ways
to choose at most $k$ of them is not larger than $N^{kq}$. If this number is
smaller than, say, $2^{0.9 N}$, then most subsets of the group cannot be
expressed as   a union of (at most) $k$ sumsets, 
providing the assertion of 
the first part of theorem \ref{tunion}. 

In order to prove the second part we also consider the more general
setting of arbitrary proper edge colorings of complete graphs. Let
$K$ be a properly edge-colored complete graph on $N$ vertices.
Suppose further that each color
appears $(1+o(1))N/2$ times. Let $\eps>0$ be a fixed small
positive real, and put $k=(2-\eps) \log_2 N$ (where as before we assume
that $N$ is sufficiently large as a function of $\eps$ and where we omit
all floor or ceiling signs when these are not crucial). Call  a subset
$A$ of vertices of $K$, $|A|=k$, a {\em rainbow clique} if all edges
of it have distinct colors. Let $R$ be a random subset of the set of colors
obtained by picking each color, randomly and independently, with 
probability $1/2$. Call a rainbow clique $A$ {\em available} if all the 
${k \choose 2}$ colors
of its edges belong to $R$, and let $C(A)$ denote the set of these
colors. For convenience call also every single edge $e$ in $R$ an available
clique and denote its color by $C(e)$.
\begin{theorem}
	\label{tcolor}
In the above notation,
	with high probability, there is a collection $\CC$ of at most
	$N \log \log N/(\log^2 N)$ available cliques so that
	$R = \cup_{A \in \CC} C(A)$.
\end{theorem}
\begin{proof}
The main part of the proof is a second moment argument that shows that
with high probability the total number of available rainbow cliques 
is close to its expectation and that with high probability almost every edge
whose color lies in $R$
belongs to roughly the same number of such available cliques. The required
collection $\CC$ can then be chosen  greedily among the
	available rainbow cliques by repeatedly adding such a clique that
	covers a maximum number of yet uncovered colors, 
	together with a smaller additional
	number of edges. We proceed with the details.

	The second moment argument resembles the one used in proof of the
	typical behaviour of the maximum clique in the random graph $G(n,1/2)$
	as described, for example, in \cite{alon2016probabilistic} section
	4.5. However, the dependence between edges of the same color 
	leads to several complications and requires some
	additional arguments.
	
Note, first, that almost every set of $k$ vertices spans a rainbow clique
in $K$. Indeed the fraction of $k$-cliques that contain two edges of the same
	color is at most $O(k^4/N)=N^{-1+o(1)}$. For each rainbow 
	$k$-clique $A$, let $X_A$ denote the indicator random variable 
	with value $1$ iff $A$ is available, that is, all the colors
	of its edges lie in $R$. Let $X=\sum X_A$ be the number of
	available $k$ cliques. Since any
	rainbow $k$-clique
	is available with probability $2^{-{k \choose 2}}$, 
	by linearity of expectation, the
	expected number of available rainbow $k$-cliques is
	$$
	\mu=(1-o(1)) {N \choose k} 2^{-{k \choose 2}}
	>\left(\frac{N 2^{-(k-1)/2}}{k}\right)^k > N^{\eps k/5}>N^{100}.
	$$

	Let $\mbox{Var}$ denote the variance of $X$. This is
	the sum of the variances of the indicator variables
	$X_A$ (which is smaller than $\mu$) plus the sum over all
	ordered pairs $A,A'$ of the covariances $Cov(X_A,X_{A'})$.
	If $A$ and $A'$ share no common color this covariance is
	$0$. Otherwise, it is at most the probability that both
	cliques $A,A'$ are available.  For each $j$, $1 \leq j \leq k-1$,
	let $\Delta_j$ denote the sum of probabilities that
	$X_A=X_{A'}=1$ where $(A,A')$ range over all ordered pairs of
	rainbow $k$-cliques in which the largest
	forest in the set of all edges of $A'$ that share colors 
	with the edges in $A$ contains $j$ edges. By the discussion
	above $\mbox{Var} \leq \mu+\Delta$, where 
	$\Delta=\sum_{j=1}^{k-1} \Delta_j$.
	The following upper bound for
	$\Delta_j$ (which can be improved) suffices for our purpose here.
	\vspace{0.2cm}

\noindent
	{\bf Claim:}\, For every $1 \leq j \leq k-1$
	$$
	\Delta_j \leq {N \choose k} {{k \choose 2} \choose j}
	{k \choose 2}^j N^{k-j} \frac{1}{(k-2j)!} 
	2^{-2{k \choose 2}+{{j+1} \choose 2}},
	$$
	where for $j>k/2$ the $(k-2j)!$ term should be replaced by $1$.
	\vspace{0.2cm}

	\noindent
	{\bf Proof of claim:}\, 
	There are at most ${N \choose k}$ ways to choose the first rainbow
	clique. Then there are  
	${{k \choose 2} \choose j}$ ways to choose $j$ colors that
	appear in it. There are less than ${k \choose 2}^j$ ways to place
	these colored edges as a forest on $k$ labelled vertices.
Once this forest is chosen, we can select the second clique by selecting
	its first vertex in each connected component of the forest 
	(including the ones of size $1$). This corresponds to the 
	$N^{k-j}$ factor. When $2j<k$ at least $k-2j$ of these first vertices
	are in components of size $1$, hence we can divide by $(k-2j)!$.
	The crucial point here is that since the 
	edge coloring is proper, this information
	suffices to reconstruct all vertices of the second rainbow clique.
	Finally, the total number of common colors between the two cliques
	is at most ${{j+1} \choose 2}$ (obtained 
	only when the forest is a tree
	and the set of edges of the second clique that share colors
	with the first forms a clique on $j+1$ vertices).
	Even in this case, the probability that all required colors 
	are in $R$ is at most $2^{-2{k \choose 2}+{{j+1} \choose 2}}$. 
	This completes
	the proof of the claim.  \hfill $\Box$

Returning to the proof of the theorem  we next show that the 
variance $\mbox{Var}$
is much smaller than $\mu^2$. Indeed, 
$$
\frac{\Delta_j}{\mu^2} \leq \left(\frac{k^6 2^{(j+1)/2}}{N}\right)^j.
$$
For small $j$, say $j \leq \log \log N$, this ratio is
less than $N^{-(1+o(1))j}$. For $\log \log N<j \leq (2-\eps) \log_2 N$
this ratio is less than $N^{-\eps j/3}$ which is much smaller than, say,
$N^{-100}$. This shows that the variance $\mbox{Var} \leq \mu + \Delta$ is
at most $\mu^2/N^{1-o(1)}$. Therefore, by Chebyshev's Inequality, with
high probability the total
number of available rainbow $k$-cliques is $(1+o(1))\mu$ where
$\mu$ is the expectation of this quantity (which is 
much larger than $N^{100}$). 

We next show that for every fixed edge of $K$, if the random set of colors
$R$ contains its color, then with high probability it lies in roughly the
expected number of available rainbow $k$-cliques that contain it. 
The computation here 
is similar to the previous one, and here too the variance
is at most the square of the expectation divided by $N^{1-o(1)}$. As the
computation is very similar, we omit it. By Chebyshev's Inequality
and Markov'v Inequality,
with high probability every color that appears in $R$  besides
$N^{o(1)}$ of them appears in roughly the same number of available rainbow
$k$-cliques, which is a $(1+o(1))(k^2/N)$ fraction of the total number of
such cliques. We can now choose greedily $3 N \log \log N /k^2$ 
available rainbow $k$ cliques one by one, where in each step
we choose such a clique that covers the maximum number of yet
uncovered colors in $R$. As in each step we cover at least
a $(1+o(1))k^2/N$ fraction of the remaining colors, this will
cover all
colors of $R$ besides $o(N/ \log^2 N)$ of them. These can be covered by
edges, completing the proof and implying also the
assertion of the second part of Theorem \ref{tunion}.
\end{proof}

\section{Concluding Remarks and Open Problems}\label{sec:summary}

Natural extensions of the main question studied in this paper include studying the cardinality of the family of sumsets in~$\F_p^n$ for~$p>2$, and studying the cardinality of the family of higher-orders of iterated sums (e.g., sets of the form~$A+A+A$ for some~$A\subseteq \F_2^n$, or generally~$kA$ for any~$k>2$).

The lower bound in Theorem \ref{tunion} can be improved to
$\Omega(N/\log^2 N)$
for some groups, like the cyclic group $Z_N$.
Indeed, Green showed in 
\cite{green2005counting} that the largest $A$ so that the sumset
$A+A$ lies in a random subset of $Z_N$ is, with high probability,
of size $O(\log N)$. This, together with the 
counting argument described in the 
proof of the first part of theorem \ref{tunion} here,
supplies the improved bound. It seems 
plausible 
that the $\log \log N$ factor can be removed in both the 
upper and the lower bounds in Theorem \ref{tunion} for every abelian 
group of order $N$. A somewhat related known result of Frieze and Reed
(\cite{FR95}), that holds in
the simpler case of the usual random graph
$G=G(N,1/2)$,  is that the minimum number of cliques required to cover
all edges of $G$ is, with high probability, $\Theta(N^2/\log^2 N)$. 

	It is worthwhile to add that much fewer sumsets suffice to express a random
subset of $\F_2^n$ as their intersection. Indeed, for a random
subset $S$ (containing $0$), define, for every $1 \leq i \leq n$,
$A_i=S \cap (\F_2^n \setminus e_i^{\perp} ) \cup \{0\}$,
where $e_i$ is the vector of Hamming  weight $1$ with its unique
$1$ coordinate in the $i$-th place.
Then $(A_i+A_i)\cap (\F_2^n \setminus e_i{\perp})= 
(S \cap (\F_2^n \setminus e_i{\perp})$, and with high probability
(that is, with probability tending to $1$ as $n$ tends to infinity)
$(A_i+A_i) \cap e_i^{\perp}=e_i^{\perp}$ for every $i$. 
It is easy to check that 
$S$ is the intersection of these $n$ sumsets $A_i+A_i$.
\vspace{0.2cm}

\noindent
{\bf Acknowledgment:}\,
We thank Matthew Jenssen for helpful comments and references.


\end{document}